\documentclass[12pt,a4paper,leqno]{amsart}

\usepackage{amsmath}
\usepackage{amssymb,hyperref}
\usepackage{multicol}
\usepackage{array,float}
\usepackage{amsmath,amsthm,amsbsy,amscd,mathrsfs}
\usepackage{graphicx}
\usepackage{mathtools}
\usepackage{cite}
\usepackage{caption}
\usepackage{xcolor}
\usepackage{tikz}

\makeatletter
\newcommand\suchthat{%
	\@ifstar
	{\mathrel{}\middle|\mathrel{}}
	{\mid}%
}
\makeatother

\newtheorem{thm}{Theorem}[section]
\newtheorem{prop}[thm]{Proposition}
\newtheorem{cor}[thm]{Corollary}
\newtheorem{lem}[thm]{Lemma}

\theoremstyle{remark}
\newtheorem{rem}[thm]{Remark}
\numberwithin{equation}{section}  
\newtheorem{defi}[thm]{Definition}

\DeclareMathOperator{\PicF}{Pic_{F}}

\DeclareMathOperator{\DivF}{Div_{F}}
\DeclareMathOperator{\To}{T^0}

\DeclareMathOperator{\Tr}{Tr}

\usepackage{mathtools}

\title{The size function for quadratic extensions of complex quadratic fields}

\author{Ha Thanh Nguyen Tran}
\address{Ha Tran\\
	Department of Mathematics and Systems Analysis,\\
	Aalto University School of Science,\\    
	Otakaari 1, 02150 Espoo\\
	Finland}
\email{hatran1104@gmail.com}



\begin{document}
	
	\begin{abstract}
		The function $h^0$ for a number field is an analogue of the dimension of the Riemann-Roch spaces of divisors on an algebraic curve. In this paper, we prove the conjecture of van der Geer and Schoof about the maximality of $h^0$ at the trivial Arakelov divisor for quadratic extensions of complex quadratic fields.
	\end{abstract}
	
	\maketitle


\bigskip
\section{Introduction}
In \cite{ref:3}, van der Geer and Schoof introduced the function $h^0$ for a number field $F$ that is also called the ``size function"  for $F$ (see \cite{ref:14,ref:15,ref:27,ref:21}). This function is well defined on the Arakelov class group $\PicF^0$ of $F$ (see \cite{ref:4}).
Van der Geer and Schoof also  conjectured concerning the maximality of $h^0$ as follows.

\textit{Conjecture}. 
Let $F$ be a number field that is Galois over $\mathbb{Q}$ or over an imaginary quadratic number field. Then the function $h^0$ on $\PicF^0$ assumes its maximum in the trivial class $O_F$.

Francini in \cite{ref:14} and \cite{ref:15} has proved this conjecture for quadratic fields and certain pure cubic fields. In this paper,  we prove that this conjecture holds for all quadratic extensions of complex quadratic fields.
\begin{thm}\label{thmmain}
	Let $F$ be a quadratic extension of a complex quadratic field. Then the function $h^0$ on $\PicF^0$ has its unique global maximum at the  trivial class $D_0=(O_F, 1)$. 
\end{thm}
Let $F$ be a quadratic extension of a complex quadratic field $K$. Recall that $\PicF^0$ is a topological group with the connected component of identity denoted by $\To$ (see Section 2). We use the condition $F$ is Galois over $K$ to show that $h^0$ is symmetric on $\To$ (see Lemma \ref{h0sym}). In general, this is not true for quartic fields that do not have any imaginary quadratic subfield. For instance, it is false and the conjecture does not hold in case of the totally complex quartic field defined by the polynomial $x^4 -x +1$ or $x^4 + x^2 -x +1$.

Since $F$ is a totally quartic fields, the group of units $O_F^*$ has rank 1. So, it has a fundamental unit $\varepsilon$. We assume that $|\varepsilon| \geq 1$. Basically, we follow the proofs of Francini (see \cite{ref:14,ref:15}). Beside that, for a quadratic extensions of a complex quadratic field, the fundamental unit $\varepsilon$ can be quite small. We need two more steps in  Section \ref{sec2b}  and Section \ref{sec4} compared with Francini's proofs. 
To prove Theorem \ref{thmmain}, we show that $h^0(D) < h^0(D_0)$ for all $D \in \PicF^0$. We distinguish two cases:  $D$ is not on $\To$ (Section 4) and  $D$ is on $\To$. In the second case, we consider separably $|\varepsilon| \geq 1 +\sqrt{2}$ (Section 5) and when $|\varepsilon| < 1 +\sqrt{2}$ (Section 6). 

For the convenience of the reader, we give a brief introduction to Arakelov divisors, $\PicF^0$  and the function $h^0$ in Section 2.

\section{Preliminaries}
In this part we briefly recall the definitions of   Arakelov divisors, the Arakelov class group and the function $h^0$ of a number field. See \cite{ref:4, ref:3} for full details. 

Let $F$ be a number field of degree $n$ and let $r_1, r_2$ the number of real and complex infinite primes (or infinite places) of $F$. Let $\Delta$ and $O_F$ be the discriminant and the ring of integers of $F$ respectively.

\subsection{Arakelov divisors}
Let $F_{\mathbb{R}}: = F \otimes_\mathbb{Q}\mathbb{R} \simeq \prod_{\sigma \text{ real}}\mathbb{R} \times \prod_{\sigma \text{ complex}}\mathbb{C}$ where $\sigma$'s are the infinite primes of $F$.  Then $F_{\mathbb{R}}$ is an \'{e}tale $\mathbb{R}$-algebra with the canonical Euclidean structure given by the scalar product
$$\langle u, v \rangle := \Tr(u \overline{v}) \text{ for any } u = (u_{\sigma}), v = (v_{\sigma}) \in F_{\mathbb{R}}.$$
The \textit{norm} of an element $u = \prod_{\sigma}u_{\sigma} $ of $F_{\mathbb{R}}$ is defined by 
$N(u):= \prod_{\sigma \text{ real} }u_{\sigma} \cdot \prod_{\sigma \text{ complex} } |u_{\sigma}|^2.$

Let $I $ be a fractional ideal of $F$. 
Each element $f$ of $I$ is mapped to the vector $(\sigma(f))_{\sigma}$ in  $F_{\mathbb{R}}$. For any vector $u$ in  $F_{\mathbb{R}}$ and $f \in I$, we have $ u f = (u_{\sigma} \sigma(f))_{\sigma} \in F_{\mathbb{R}}$, so  
$\|u f\|^2 = \sum_{\sigma  } deg(\sigma) u_{\sigma}^2 |\sigma(f)|^2 .$
Here $deg(\sigma)$ is equal to $1$ or $2$ depending on whether $\sigma$ is real or complex.

\begin{defi}
	An \textit{ Arakelov divisor}  is a pair $D=(I,u)$ where $I$ is a fractional ideal and $u$ is an arbitrary unit in $\prod_{\sigma}\mathbb{R}^*_{+} \subset F_{\mathbb{R}}$. 
\end{defi}
All of Arakelov divisors of $F$ form an additive group denoted by $\DivF$. 
The \textit{degree} of $D = (I,u)$ is defined by $deg(D): = \log{N(u) N(I)}$. 
We associate to $D$ the \textit{lattice} $ uI =\{u x: x \in I\}\subset F_{\mathbb{R}} $ with the metric inherited from $F_{\mathbb{R}}$ (see about ideal lattices in \cite{ref:0}). For each $f \in I$,  by putting $\| f\|_D:= \| uf\|$, we obtain a  scalar product on $I$ that makes $I$ an ideal lattice as well {\cite[Section 4]{ref:4}}. 
To each element $f \in F^*$ is attached a \textit{principal} Arakelov divisor $(f)=(f^{-1} O_F, |f|)$ where 
$f^{-1} O_F$ is the principal ideal generated by $f^{-1}$ and $|f| = (|\sigma(f)|)_{\sigma} \in F_{\mathbb{R}}$. It has degree $0$ by the product formula.

\subsection{The Arakelov class group}\qquad \\
The set of all Arakelov divisors of degree 0 form a group, denoted by $\DivF^0$. 
Similar to the Picard group of an algebraic curve, we have the following definition.
\begin{defi}
	The \textit{Arakelov class group }  $\PicF^0$ is the quotient of $\DivF^0$ by its subgroup  of principal divisors.
\end{defi}

Each $v=(v_{\sigma})  \in \oplus_{\sigma}\mathbb{R}$ can be embedded into $\DivF$ as the divisor $D_v= (O_F, u)$ with $u = (e^{-v_{\sigma}})_{\sigma}$.
Denote by 
$(\oplus_{\sigma}\mathbb{R})^0 = \{ (v_{\sigma}) \in \oplus_{\sigma}\mathbb{R}: deg(D_v) = 0 \}$ and $\Lambda = \{(log|\sigma(\varepsilon)|)_{\sigma}: \varepsilon \in O_F^*\}$. Then $\Lambda$ is a lattice contained in  the vector space $(\oplus_{\sigma}\mathbb{R})^0$. We define
$$\To = (\oplus_{\sigma}\mathbb{R})^0 / \Lambda.$$
By Dirichlet's unit theorem, $\To$ is a compact real torus of dimension $r_1+r_2-1$ {\cite[Section 4.9]{ref:11}}. Denoting by $Cl_F$ the class group of $F$, the structure of $\PicF^0$ can be seen by the following proposition.
\begin{prop}\label{prop:structure1}
	The map that sends each class of divisor  $(I, u)$ to the  class of ideal $I$  is a homomorphism 
	from $ \PicF^0$ to the class group $Cl_F$ of $F$. 
	It induces the exact sequence 
	\begin{align*}
		0 \longrightarrow \To \longrightarrow \PicF^0 \longrightarrow Cl_F \longrightarrow 0  .
	\end{align*}
\end{prop}
\begin{proof}
	See Proposition 2.2 in \cite{ref:4}.
\end{proof}
Thus, the group $\To$ is the connected component of the identity of the topological group $\PicF^0$. 
Each class of Arakelov divisors in $\To$ is represented by a divisor of the form $D = (O_F, u)$ for some  
$  u \in \prod_{\sigma}\mathbb{R}^*_{+} $.  Here $u$ is unique up to multiplication by units $\varepsilon \in O_F^*$ {\cite[Section 6]{ref:4}}.

\subsection{The function $h^0$ of a number field}\qquad \\
Let $D=(I,u)$ be an Arakelov divisor of $F$. 
We denote by  $$k^0(D) = \sum_{f \in I}e^{-\pi\|f\|^2_{D}} \hspace*{1cm}  \text{ and } \hspace*{1cm}  h^0(D)=\log(k^0(D)).$$ 
The function $h^0$ is  well defined on $\PicF^0$ and analogous to the dimension of the Riemann-Roch space $H^0(D)$ of a divisor $D$ on an algebraic curve. See \cite{ref:3} for full details.

\section{Some results}\label{sec1}
From now on, we fix a quadratic extension $F$ of some complex quadratic field $K$. 
Let  $\tau: F \longrightarrow F$ be the automorphism of $F$ that generates $Gal(F|K)$. Assume that $F = \mathbb{Q}(\beta) $ for some $\beta \in F$. We denote by $\sigma: \beta \longmapsto \beta$ an  infinite prime of $F$. Then  $\sigma'= \sigma \circ \tau$ is the second infinite prime. Moreover, we identify $F$ with $\sigma(F)$ in this paper.

Let $D= (I, u)$ be an Arakelov divisor of degree $0$ of $F$ with $L=u I$ the ideal lattice associated to $D$. We denote by $\lambda$  the length of the shortest vectors of $L$.

Denote by $B^t = \{f \in L: M \leq \|f\|^2 \leq t \}$ for each $t> M \geq \lambda^2$.
We first prove the following lemma.

\begin{lem}\label{Bt}  For each $t> M \geq \lambda^2 \geq a^2$ with $a>0$, we have
	$$\#B^t \leq \left( \frac{2\sqrt{t}}{a} +1\right)^4 - \left( \frac{2\sqrt{M}}{a} -1\right)^4.$$
\end{lem}
\begin{proof}
	Let  $B^t = \{f \in L: M \leq \|f\|^2 \leq t \}$ for each $t> M$. The balls with centers in $x \in B^t$ and radius $\lambda/2$ are disjoint. Their union is contained in the (hyper) annular disk 
	$$\{ x \in F_{\mathbb{R}}: \sqrt{M} - \lambda/2 \leq \|x\| \leq  \sqrt{t} +  \lambda/2\}.$$
	By computing their volumes, we get that 
	$$ \left(\frac{\lambda}{2}\right)^4 \#B_t \leq  \left(\sqrt{t} +\frac{\lambda}{2}\right)^4 - \left(\sqrt{M} -\frac{\lambda}{2}\right)^4.$$
	Dividing by  $\left(\frac{\lambda}{2}\right)^4$, we get 
	$$\#B^t \leq \left( \frac{2\sqrt{t}}{\lambda} +1\right)^4 - \left( \frac{2\sqrt{M}}{\lambda} -1\right)^4.$$
	Since this bound for $\#B^t$ is a decreasing function in $\lambda$ and $\lambda \geq a$, the lemma is proved.
\end{proof}

\begin{lem}\label{err}
	Let $M \geq \lambda^2\geq a^2>0$ with $a>0$.  	
	Then  
	$$\sum_{\substack{ f \in L   \\ \|f\|^2 \geq  M }}e^{-\pi \|f\|^2 } \leq \pi  \int_{M}^{\infty} \!  \left( \left( \frac{2\sqrt{t}}{a} +1\right)^4 - \left( \frac{2\sqrt{M}}{a} -1\right)^4 \right) e^{- \pi t}\, \mathrm{d}t.$$
\end{lem}

\begin{proof}
	For each $t> M$, denote by $S$ the sum on the left side of the lemma, we have  
	$$S = \sum_{\substack{ f \in L  \\ \|f\|^2 \geq M }}\int_{\|f\|^2}^{\infty} \! \pi e^{- \pi t}\, \mathrm{d}t  \leq   \pi  \int_{M}^{\infty} \! \# B^t e^{- \pi t}\, \mathrm{d}t.$$
	Using Lemma \ref{Bt}, we get the result.
\end{proof}

\begin{cor}\label{sum}
	Assume  $\lambda^2  \geq 4$. Then we have 
	$$\sum_{\substack{ f \in L   \\ \|f\|^2 \geq 4\sqrt{2} }}e^{-\pi \|u x\|^2 } < 2.6729 \cdot 10^{-6} 
	\hspace*{1cm} \text{ and }  \sum_{\substack{ f \in L   \\ \|f\|^2 \geq 4\sqrt{3} }}e^{-\pi \|u x\|^2 } < 6.3067 \cdot 10^{-8}.$$
	
\end{cor}
\begin{proof}
	Use Lemma \ref{err} with $a=2$, $M=  4 \sqrt{2} $ for the first sum and $M=  4 \sqrt{3} $ for the second sum.
\end{proof}

Let $D = (O_F,u)$ be an Arakelov divisor of degree $0$. Then $N(u) = 1$ and so $u$ has the form $(s, 1/s)$ for some $s \in \mathbb{R}_{+}$.  
Let $x \in O_F \backslash \{0\}$. Then $\|u x\|^2 = 2s^2 |x|^2 + 2|\sigma'(x)|^2/s^2$ and $N(u x) = |x|^2 \|\sigma'(x)|^2 = N(x) >0$. Therefore, we have that $\|u x\|^2 = 2s^2 |x|^2 + 2N(x)/(s^2 |x|^2 )$.

\begin{lem} \label{h0sym}
	Let $F$ be a quadratic extension of some complex quadratic field $K$. Then $h^0$ is symmetric on $\To$.
\end{lem}
\begin{proof}
	Let $D = (O_F, u) \in \To$ with $u=(s, 1/s)$ for some $s \in \mathbb{R}_{+}$. Let  $\tau$ be the automorphism of $F$ that generates $Gal(F|K)$. Then $\tau$ switches the infinite primes of $F$. Therefore, $\tau(D)= \tau((O_F, (s, 1/s)))=(O_F, ( 1/s, s))=-D$. So $\|x\|_D^2= \|\tau(x)\|_{\tau(D)}^2$ for all $x \in O_F$. Thus, the lattices associated to $D$ and $\tau(D)$ are isometric {\cite[Section 4]{ref:4}}. Hence, $k^0(D) = k^0(-D)$.
\end{proof}

For each $j=2,3$ and $s \in [0.8722, 1.1465] $, we denote by  $$\mathfrak{B}_j(s)= \{x \in O_F :  |N(x)| = j \text{ and } \|u x\|^2 < 8\}.$$
Then we have the following results. 

\begin{lem}\label{lengthB}
	Let $x \in \mathfrak{B}_j(s)$ for $j=2,3$. Then  $\|x\|^2 < 11$ for all $s \in [0.8722, 1/0.8722] $.
\end{lem}
\begin{proof}
	We have $\|u x\|^2 = 2 s^2 |x|^2 + 2 N(x)/( s^2 |x|^2)\geq \|x\|^2 \times 0.8722^2 $ since $s \in [0.8722, 1/0.8722] $. If $x \in\mathfrak{B}_j(s)$  then $\|u x\|^2 < 8 $. Hence $\|x\|^2 < 8/ 0.8722^2 < 11$.
\end{proof}

\begin{prop}\label{proB} 
	Assume that $F$ has a fundamental unit $\varepsilon$ with $|\varepsilon| \geq 1 + \sqrt{2}$.  
	Then for all $s \in [0.8722, 1/0.8722] $, each set $\mathfrak{B}_2(s)$ and $\mathfrak{B}_3(s)$ has at most $30$ elements.
\end{prop}

\begin{proof} 
	For each $j=2,3$, let $m_j= \#\mathfrak{B}_j(s)$. All elements in $\mathfrak{B}_j(s)$ generate some prime ideal of norm $j$. Since there are at most 4 ideals of norm $j$, this means that $m_j/4$ of those elements generate the same ideal.
	This implies that their quotients are units. So there are $m_j/4$ different units.
	But the unit group is generated by $\varepsilon$ and $\omega$ roots 
	of unity. This means that one of those $m_j/4$ units, say $\varepsilon_1$,  must be
	$\pm \varepsilon^k$ with $k > \frac{m_j}{4 \omega}$.
	
	But  $k$ cannot be too large, because  $\varepsilon$ is the quotient of two small elements $x$ and $y$ in $\mathfrak{B}_j(s)$. We have 
	\begin{multline*}
		\| x/y\|^2 = 2 |x/y|^2 + 2 |\sigma'(x)/\sigma'(y)^2| \leq (2 |x|^2 + 2 |\sigma'(x)|^2) ( 1/|y|^2 +  1/|\sigma'(y)|^2)\\
		= \frac{1}{2 j}\|x\|^2 \|y\|^2.
	\end{multline*}
	The last equality is because $1/|y|^2 +  1/|\sigma'(y)|^2 = (2|y|^2 +  2|\sigma'(y)|^2)/(2 N(y))$ and $N(y) =j$. In fact, for each $j=2,3$, we know $\frac{1}{2 j}\|x\|^2 \|y\|^2<\frac{11^2}{4} $ by Lemma \ref{lengthB}. Then 
	\begin{equation}\label{eqk}
		2 |\varepsilon|^{2k} +  \frac{2}{|\varepsilon|^{2k}} = \|\varepsilon^k\|^2 =\| x/y\|^2< \frac{1}{2 j}\|x\|^2 \|y\|^2< \frac{11^2}{4}.
	\end{equation}
	
	Since $|\varepsilon| \geq 1 +\sqrt{2}$, the inequality in \eqref{eqk} implies that $k \leq 1$.  Moreover, it is known that $F$ has at most 8 roots of unity since the fundamental unit $|\varepsilon| \geq 1 +\sqrt{2}$. So $\omega \leq 8$. This and the inequality $\frac{m_j}{4 \omega} <k \leq 1$ lead to $m_j <32$. Since the number of elements in $\mathfrak{B}_j(s)$ is always even, $\mathfrak{B}_j(s)$ has at most 30 elements. 
	
\end{proof}

\section{Case 1: $D$ is not on $\To$}
\begin{prop}\label{pro1}
	Let $D $ be a class of Arakelov divisors in $\PicF^0$. If $D$ is not on $\To$ then $k^0(D) < k^0(D_0)$ where $D_0 = (O_F,1)$ is the trivial divisor.
\end{prop}
\begin{proof}
	Since $D$ is not on $\To$, we can assume that $D$ has the form $(I,u)$ where $I$ is not principal and $u \in (\mathbb{R}^*_{+})^2$. 
	
	Let $x \in  I \backslash \{0\}$. Then $\frac{|N(x)|}{N(I)} \geq 2$ because $I$ is not principal. In addition, $deg(D) = 0$, so $N(I) N(u) =1$. Therefore
	$$\|u x\|^2 \geq 4 |N(u x)|^{2/4} = 4 |N(u) N( x)|^{1/2} = 4 \left(\frac{|N(x)|}{N(I)}\right)^{1/2} \geq 4 \sqrt{2}.$$
	Hence, we obtain the following.
	$$k^0(D) = 1 + \sum_{\substack{ x \in I\backslash \{0\}   }}e^{-\pi \|u x\|^2 } = 1+ \sum_{\substack{ f \in u I\backslash \{0\} \\ \|f\|^2 \geq 4 \sqrt{2}   }}e^{-\pi \|f\|^2 }.$$
	and $ \lambda^2 \geq 4 \sqrt{2}$ where $\lambda$ is the length of the shortest vectors of the lattice $u I$. 
	Corollary \ref{sum} implies that 
	$$ k^0(D)< 1+2.67287 \cdot 10^{-6}.$$
	On the other hand, we have 
	$$k^0(D_0) > 1 + 2 e^{-4 \pi} > 1 + 6.9 \cdot 10^{-6}.$$
	Thus, $k^0(D_0) > k^0(D)$.
\end{proof}

\section{Case 2: $D$ is  on $\To$ and $|\varepsilon| \geq 1+ \sqrt{2}$}\label{sec3}

We can assume that $F$ has a fundamental unit 
$\varepsilon$ for which  $|\varepsilon| >1$. From now on, we fix this $\varepsilon$. 

Let $D = (O_F,u) \in \To$. Here $u$ has the form $(s, 1/s)$ for some $s \in \mathbb{R}^*_{+}$. By the definition of $\To$ (Section 2.2), it is sufficient to consider the case in which $ s \in [|\varepsilon|^{-1/2}, |\varepsilon|^{1/2}]$. 
We have three cases.

\subsection{Case 2a:  $s \in [|\varepsilon|^{-1/2}, 0.8722) \cup (1.1465, |\varepsilon|^{1/2}]$}\label{sec2a}
\begin{prop} \label{pro2a}
	If $D = (O_F,u)$ is on $\To$ where $u = (s , s^{-1})$ and $s \in [|\varepsilon|^{-1/2}, 0.8722) \cup (1.1465, |\varepsilon|^{1/2}]$ then $k^0(D) < k^0(D_0)$.
\end{prop}
\begin{proof}
	We have $k^0(D) = S_{1} + S'_{1} $ with 
	$$ S_{1} = \sum_{\substack{ x \in O_F \\ \| u x\|^2 < 4\sqrt{2}   }}e^{-\pi \|u x\|^2 } \hspace*{0.5cm}\text{ and } \hspace*{0.5cm} S'_{1} = \sum_{\substack{ x \in O_F \\ \| u x\|^2 \geq 4 \sqrt{2}   }}e^{-\pi \|u x\|^2 }.$$
	Let $x \in  O_F \backslash \{0\}$. Then $N(u x) = N(x) \geq 1$ since $N(u)=1$. We have
	$$\|u x\|^2 \geq 4 |N(u x)|^{2/4} = 4 |N(x)|^{1/2} \geq 4.$$
	Thus, $ \lambda^2 \geq 4 $ where $\lambda$ is the length of the shortest vectors of the lattice $u O_F$.
	Corollary \ref{sum} says that 
	$S'_{1} < 2.673 \cdot 10^{-6}$.\\
	Now let $x \in O_F \backslash \{0\}$ such that $\|u x\|^2 < 4 \sqrt{2}$. Then we must have $|N(x)|=1$. So $ x =  \zeta \cdot \varepsilon^m$ for some integer $m$ and some root of unity $\zeta$ of $F$. If $|m|\geq 1$ then $\|u \varepsilon^m\|^2 \geq 4 \sqrt{2}$. Hence $m = 0$, so $x$ is a root of unity of $F$. Then so $S_{1} \leq 1 + \omega \cdot e^{ \|u\|^2} = 1 + \omega \cdot e^{-\pi( 2 s^2 + 2/s^2)}$ where $\omega$ is the number of roots of unity of $F$. For $\omega\geq 2$, we obtain that 
	$$k^0(D) \leq 1 + \omega \cdot e^{-\pi( 2 s^2 + 2/s^2)}+ 2.673 \cdot 10^{-6}  \leq 1 + \omega \cdot e^{-4 \pi}  $$
	for all $ s \in [|\varepsilon|^{-1/2}, 0.8722) \cup (1.1465, |\varepsilon|^{1/2}].$\\
	Since  $k^0(D_0) >  1 + \omega \cdot e^{-4 \pi}$, we get $ k^0(D_0)>  k^0(D)$.
\end{proof}

\subsection{Case 2b:  $s \in [0.8722, 0.9402) \cup (1.0637, 1.1465]$}\label{sec2b}
\begin{prop} \label{pro2b} 
	If $D = (O_F,u)$ is on $\To$ where $u = (s , s^{-1}) $ and $s \in [0.8722, 0.9402) \cup (1.0637, 1.1465] $ then $k^0(D) < k^0(D_0)$.
\end{prop}

\begin{proof}
	We have $k^0(D) = S_{1} + S_2 + S'_2 $ where 
	$$ S_2 = \sum_{\substack{ x \in O_F \\ 4\sqrt{2} \leq \| u x\|^2 < 4\sqrt{3}   }}e^{-\pi \|u x\|^2 } \hspace*{1cm},  \hspace*{1cm} S'_2 = \sum_{\substack{ x \in O_F \\ \| u x\|^2 \geq 4 \sqrt{3}   }}e^{-\pi \|u x\|^2 }, $$
	and $S_{1} $ as in the proof of Proposition \ref{pro2a} and $S_1 \leq 1 + \omega \cdot e^{-\pi( 2 s^2 + 2/s^2)}$. 
	
	By Corollary \ref{sum}, we obtain that 
	$S'_2 < 6.3067 \cdot 10^{-8}$.\\
	Now we compute $S_2$. Let $x \in O_F \backslash \{0\}$ such that $4 \sqrt{2} \leq \|u x\|^2 < 4 \sqrt{3}$. Then  $|N(x)|$ is equal to 1 or 2. We claim that $|N(x)| \neq 1 $. Indeed, if not then $ x = \zeta \cdot \varepsilon^m$ for some integer $m$ and some root of unity $\zeta$ of $F$. If $m\neq 0$ then  $\|u \varepsilon^m\|^2 > 4 \sqrt{3}$ (since $|\varepsilon|^2 \cdot 0.8722^2 \geq  (1 + \sqrt{2})^2 \cdot 0.8722^2 > \sqrt{2} + \sqrt{3}$) and if $m=0$ then  $\|u x\|^2 = \|u\|^2 < 4 \sqrt{2} $ for all $s \in [0.8722, 0.9546) \cup (1.0476, 1.1465]$. This contradicts the fact that $4 \sqrt{2} \leq \|u x\|^2 < 4 \sqrt{3}$. Thus, $|N(x)| =2 $. By Proposition \ref{proB}, there are at most 30 possibilities for $x$. Therefore
	$$S_2 \leq 30 \max_{\substack{ x \in O_F \\ 4\sqrt{2} \leq \| u x\|^2 < 4\sqrt{3}   }} e^{-\pi \|u x\|^2} \leq 30 e^{- 4\sqrt{2} \pi}.$$
	Then
	$$k^0(D) \leq 1 +  \omega \cdot e^{-\pi( 2 s^2 + 2/s^2)}+ 30 e^{- 4\sqrt{2} \pi} + 6.3067 \cdot 10^{-8}  \leq 1 +\omega \cdot e^{-4 \pi} $$
	for all $s \in [0.8722, 0.9402) \cup (1.0637, 1.1465]  $ and all $\omega \geq 2$.  
	Since  $k^0(D_0) >  1 + \omega \cdot e^{-4 \pi}$, the result follows.
\end{proof}

\subsection{Case 2c:  $s \in [0.9402, 1.0637]$}\label{sec2c}
\qquad\\ Let $D = (O_F, u)$ be an Arakelov divisor of degree 0 with $u= (s, 1/s)$. 

For each $m \in \mathbb{Z}_{\geq 1}$, denote by 
$$B_m = \{x \in O_F: 4 \sqrt{m} \leq \| u x\|^2 < 4 \sqrt{m+1} \}.$$
It is clear that $ N(x) \leq m$ for all $x \in B_m$ because we know that $\|u x\|^2 \geq 4 N(u x)^{1/2} = 4 N(x)^{1/2} $ {\cite[Proposition 3.1]{ref:4}}.
Now let 
$$g(s) = k^0(D) =\sum_{ x \in O_F }e^{-\pi (\|u x\|^2) }= \sum_{ x \in O_F }e^{-\pi (2s^2 |x|^2 + 2N(x)/(s^2|x|^2) }$$
for all   $s >0$.  
We prove that this function has its maxima at $s=1$ on the interval $[0.9402, 1.0637]$. In other words, we prove the following.

\begin{prop}\label{pro2c}
	We have $g'(1) =0$ and $g''(s) <0$ for all   $s \in [0.9402, 1.0637]$.
\end{prop}
\begin{proof}
	Let $s>0$ and denote by $D = (O_F, (s,1/s))$. Then $D\in \To$ and so $k^0(D)= k^0(-D)$ by Lemma \ref{h0sym}. Hence we have $g(s) = g(1/s)$. This implies that $g'(1) =0$, the first statement is proved.
	
	Take the second derivative of $g$, we get
	$$g''(s) = \frac{4 \pi}{s^2} \sum_{x \in O_F \backslash \{0\}}G(s,x) $$
	where $$G(s,x) = \left(\pi \|u x\|^4 -16 \pi N(x) - \frac{\|u x\|^2}{2} -\frac{2 N(x)}{s^2 |x|^2} \right) e^{-\pi \|u x\|^2}.$$
	
	Let $$T_i = \sum_{x \in B_i} G(s,x) \text{ for } i =1, 2, 3  \hspace*{1cm}\text{  and  } \hspace*{1cm} T_4 = \sum_{\substack{ x \in O_F   \\ \|u x\|^2 \geq 8 }} G(s,x) .$$
	Then $g''(s) = \frac{4 \pi}{s^2} (T_1 + T_2 + T_3 + T_4) $ because $\| u x\| \geq 4 $ for all $x \in O_F\backslash\{0\}$. Therefore, in order to prove $g''(s)<0$ we show that $ T_1+ T_2+ T_3 + T_4 < 0$. This follows from Lemma \ref{lemT4}, \ref{lemT3},  \ref{lemT2} and \ref{lemT1} below.
\end{proof}

\begin{lem}\label{lemT4}
	For all $s \in [0.9402, 1.0637] $, we have
	$$T_4 < 3.9 \cdot 10^{-7}.$$
\end{lem}
\begin{proof}
	We have 
	$$T_4 \leq \sum_{\substack{ x \in O_F   \\ \|u x\|^2 \geq 8 }} \left(\pi \|u x\|^4 -16 \pi N(x) - \frac{\|u x\|^2}{2}  \right) e^{-\pi \|u x\|^2}.$$
	Therefore
	\begin{multline*}
		T_4 \leq \sum_{\substack{ x \in O_F  \\ \|u x\|^2 \geq 8 }}\int_{\| u x\|^2}^{\infty} \! \left(\pi^2 t^2 - \frac{5 \pi t}{2} - 16 \pi^2 +\frac{1}{2}\right) e^{- \pi t}\, \mathrm{d}t \\
		\leq   \pi  \int_{8}^{\infty} \! \# B^t \left(\pi^2 t^2 - \frac{5 \pi t}{2} - 16 \pi^2 +\frac{1}{2}\right) e^{- \pi t}\, \mathrm{d}t
	\end{multline*}
	Since the shortest vectors of the lattice $ u O_F$ have length $\lambda \geq 2$, Lemma \ref{Bt} says that 
	$$\#B^t \leq \left( \frac{2\sqrt{t}}{2} +1\right)^4 - \left( \frac{2\sqrt{8}}{2} -1\right)^4 = \left( \sqrt{t} +1\right)^4 - \left( \sqrt{8} -1\right)^4.$$
	Replace this bound for $\#B^t$ to the last integral and compute it, we obtain the result.
\end{proof}

\begin{lem}\label{lemT3}
	For all $s \in [0.9402, 1.0637] $, we have
	$$T_3 < 5.2 \cdot 10^{-7}.$$
\end{lem}
\begin{proof}
	Let $x \in B_3$. It is easy to see that $N(x) \neq 1$ (see the proof of Proposition \ref{pro2a}), so $N(x)$ is equal to 2 or 3. In other words, we have $B_3 \subset \mathfrak{B}_2(s) \cup \mathfrak{B}_3(s)$.\\ 
	
	If $N(x) = 2$ then $\|u x\|^2 = 2s^2 |x|^2 + 4/(s^2 |x|^2 )$. Let $z = s^2 |x|^2 $. Since $ 4\sqrt{3} \leq \|u x\|^2 < 8$, we have $z \in (2- \sqrt{2},\sqrt{3}-1]  \cup [\sqrt{3}+1, 2 + \sqrt{2})$. Then for all $z$ in this interval, we have  
	$$G(s,x) = \left((2 z + 4/z)^2 -32 z -1/2(2 z + 4/z)  -4/z \right) e^{-\pi (2 z + 4/z)} \leq 1.6 \cdot 10^{-8}.$$
	
	If $N(x) = 3$ then $\|u x\|^2 = 2s^2 |x|^2 + 6/(s^2 |x|^2 )$. Let $z = s^2 |x|^2 $. Since $ 4\sqrt{3} \leq \|u x\|^2 < 8$, we get $z \in (1, 3)$. Then for all $z$ in this interval, we have 
	$$G(s,x)  = \left( (2 z + 6/z)^2 -32 z -1/2(2 z + 6/z)  -6/z \right)e^{-\pi (2 z + 6/z)} \leq 1.3 \cdot 10^{-9}.$$
	
	Proposition \ref{proB} says that $B_3$ has at most 30 elements of norm 2 and at most 30 elements of norm 3. Thus, 
	$$T_3 \leq 30 \cdot 1.6 \cdot 10^{-8} + 30 \cdot 1.3 \cdot 10^{-9} < 5.2 \cdot 10^{-7}.$$
\end{proof}

\begin{lem}\label{lemT2}
	For all $s \in [0.9402, 1.0637] $, we have
	$$T_2 < 1.65 \cdot 10^{-6}.$$
\end{lem}

\begin{proof}
	Let $x \in B_2$. Then $N(x) \leq 2$. By an argument similar to the proof of Proposition \ref{pro2a}, we obtain that   $N(x) \neq 1$, so $N(x)= 2$. Therefore $B_2 \subset \mathfrak{B}_2(s)$. Proposition \ref{proB} says that $\#B_2 \leq  \#\mathfrak{B}_2(s) \leq 30$.
	
	Let $z = s^2 |x|^2 $. Then $z \in (\sqrt{3}-1, \sqrt{3}+1)$ since $ 4\sqrt{2} \leq \|u x\|^2 < 4 \sqrt{3}$. Then so 
	$$G(s,x) = \left((2 z + 4/z)^2 -32 z -1/2(2 z + 4/z)  -4/z \right) e^{-\pi (2 z + 4/z)} < 5.5 \cdot 10^{-8}.$$
	Thus, 
	$T_2 \leq \#B_2 \cdot \max_{x \in B_2} G(s,x) < 30 \cdot 5.5 \cdot 10^{-8} = 1.65  \cdot 10^{-6}.$
\end{proof}

\begin{lem}\label{lemT1}
	For all $s \in [0.9402, 1.0637] $, we have
	$$T_1 < -2.22 \cdot 10^{-5}.$$
\end{lem}
\begin{proof}
	Let $x \in B_1$. Then we  have  $N(x) = 1$. As the proof of Proposition \ref{pro2a}, we have  $x$ is a root of unity of $F$. So 
	$T_1 =  \omega \cdot G(s,1) < -2.22 \cdot 10^{-5}$ for all $s \in [0.9402, 1.0637]$ and all $\omega \geq 2$.

\end{proof}

\section{Case 3: $D$ is  on $\To$ and $|\varepsilon| < 1+ \sqrt{2}$}\label{sec4}
With the notations in Section \ref{sec2c}, it is obvious to see the following lemma.
\begin{lem}\label{pos}
	Let $x \in O_F$. Then for all  $s \in [0.98, 1/0.98]$, we have $G(s,x) > 0$ if $e^{0.54/2} \leq |x| \leq 1+\sqrt{2}$ and $G(s,x) < 0$ if $|x|=1$.
\end{lem}
We consider 2 cases: When $\varepsilon$ does not generate $F$ and when $\varepsilon$ generates $F$.

\subsection{Case 3a: $\varepsilon$ does not generate $F$}
We prove the following proposition.
\begin{prop}\label{pro3a}
	Let $F$ be a quadratic extension of some complex quadratic subfield. Assume that $F$ has a  fundamental unit $\varepsilon$ that does not generate $F$  and $|\varepsilon|< 1+\sqrt{2}$. Then $k^0 $ has its unique maximum at the trivial divisor $D_0$ on $\To$.
\end{prop}

We first prove the lemma below.

\begin{lem}\label{case3a}
	Let $F$ be a quadratic extension of some complex quadratic subfield. Assume that $F$ has the fundamental unit $\varepsilon$ that does not generate $F$ and $|\varepsilon| < 1 + \sqrt{2}$. Then $F$ contains the quadratic subfield $K= Q(\sqrt{5})$ and $\varepsilon = (1+\sqrt{5})/2$. In particular, $O_F$ has no elements of norm 2 or 3.
	
\end{lem}

\begin{proof}
	The assumption that $\varepsilon$ does not generate $F$ implies that $K= Q(\varepsilon)$ is a real quadratic subfield of $F$. Let $\Delta_K$ be the discriminant of $K$. Then 
	$$|\varepsilon| \geq \frac{\sqrt{\Delta_K}+\sqrt{\Delta_K-4}}{2}.$$
	See \cite{ref:29}. Since $|\varepsilon| < 1 + \sqrt{2}$, we must have $4 \leq \Delta_K \leq 7$. It is easy to check that $K=\mathbb{Q}(\sqrt{5})$ and $\varepsilon = (1+\sqrt{5})/2$. So the first statement is proved.
	
	Now we suppose that there is an element element $x$ of norm 2 or 3  in $O_F$. Then $y= N_{F/K}(x)$ is in the ring of integers $O_K $ of $K$ and $N_{K/\mathbb{Q}}(y) \in  \{2,3\}$. This is impossible because $2$ and $3$ are inert in $O_K$. Thus, the second statement follows. 
	
\end{proof}

We now prove Proposition \ref{pro3a}.
\begin{proof} 
	By Lemma \ref{case3a}, we have $\varepsilon = (1+\sqrt{5})/2$. With the notations in Section \ref{sec3}, we prove this proposition in 3 steps as Proposition \ref{pro2a}, \ref{pro2b} and \ref{pro2c} respectively.
	\begin{itemize}
		\item Step 1: Let $s \in [|\varepsilon|^{-1/2}, 0.8608) \cup (1.1618, |\varepsilon|^{1/2}]$. Then using the same proof as Proposition \ref{pro2a}, we have $k^0(D) < k^0(D_0)$.
		
		\item Step 2: Let $s \in [0.8608, 0.9770) \cup (1.0235, 1.1618]$. By Lemma \ref{case3a}, there are no elements of norm 2 in $B_2$. So, $B_1 $ and $B_2$ only contain elements of norm 1. Hence $B_1 \cup B_2 \subset \{ \zeta, \zeta \cdot \varepsilon, \zeta \cdot \varepsilon^{-1} \}$ where $\zeta$ runs over the roots of unity of $F$. This leads to
		$$S_1 + S_2 \leq 1 + \omega \cdot ( e^{-\pi \|u \|^2}+e^{-\pi \|u \varepsilon\|^2} + e^{-\pi \|u \varepsilon^{-1}\|^2}) $$ for all $s \in [0.8608, 0.9770) \cup (1.0235, 1.1618]$.
		It is easy to check that for all $s$ in this interval and $\omega \geq 2$, we get
		\begin{multline*}
			\hspace*{0.8cm}1 + \omega \cdot ( e^{-\pi \|u \|^2}+e^{-\pi \|u \varepsilon\|^2} + e^{-\pi \|u \varepsilon^{-1}\|^2}) + 6.31\cdot 10^{-8}\\
			\leq 1 + \omega \cdot e^{-4 \pi}.
		\end{multline*}
		
		Since $S'_2 < 6.31 \cdot 10^{-8}$ by Corollary \ref{sum}, we obtain that
		$$k^0(D) = S_{1} + S_2 + S'_2 \leq 1 + \omega \cdot e^{-4 \pi} < k^0(D_0).$$
		
		\item Step 3: We prove that $g''(s)<0$ for $s \in [0.9770, 1.0235]$. Lemma \ref{case3a} says that there are no elements of norm 2 or 3 in $B_1$, $B_2$ and $B_3$. So their union is contained in $\{ \zeta, \zeta \cdot \varepsilon, \zeta \cdot \varepsilon^{-1} \}$ where $\zeta$ runs over the roots of unity of $F$. In addition, $1 \in B_1$ for every $ s \in [0.9770, 1.0235]$. By Lemma \ref{pos}, we obtain the following. 
		$$T_1+ T_2 + T_3 \leq \omega \cdot (G(s, 1)+ G(s, \varepsilon)+ G(s, \varepsilon^{-1}))< -2.4 \cdot 10^{-5}.$$
		We have $T_4 < 3.9 \cdot 10^{-7} $ by Lemma \ref{lemT4}, so  
		$g''(s)= T_1+T_2 + T_3 + T_4 < 0$ for all $s \in [0.9770, 1.0235]$.

	\end{itemize}
	
\end{proof}

\subsection{Case 3b: $\varepsilon$ generates $F$}
We prove the following proposition.
\begin{prop}\label{pro3b}
	Let $F$ be a quadratic extension of some complex quadratic subfield. Assume that $F$ has a fundamental unit $\varepsilon$ that generates $F$ and $|\varepsilon|< 1+\sqrt{2}$. Then $k^0 $ has its maxima at the trivial divisor $D_0$ on $\To$.
\end{prop}

First, we prove the following results.
\begin{lem}\label{bbddisc1}
	Let $F$ be a quadratic extension of some complex quadratic subfield. Assume that $F$ has the fundamental unit $\varepsilon$ that generates $F$ with $|\varepsilon| < 1 + \sqrt{2}$. Then the discriminant of $F$ is no more than 16384.
	
\end{lem}

\begin{proof}
	Since $\varepsilon$ has norm 1, we can assume that its conjugates have the form 
	$a e^{i t_1}$, $a e^{-i t_1}$, $\frac{1}{a} e^{i t_2}$ and $\frac{1}{a} e^{-i t_2} $ where $a= |\varepsilon| < 1 + \sqrt{2}$. Let $ A = \frac{1}{2}(a^2 + 1/a^2)$. Then we have $1 \leq A \leq 3$. 
	
	Because $\varepsilon$ generates $F$, the set $\{1, \varepsilon, \varepsilon^2, \varepsilon^{-1}\}$ contains linearly independent elements of $O_F$. So, the discriminant of this set is nonzero and at least the discriminant of $F$. Thus, we have that
	$$\Delta_F \leq 16^2(1 - X^2)( 1 - Y^2)(X^2 + Y^2 - 2 A X Y + A^2 - 1)^2= f(X,Y)$$ where $X=\cos(t_1)$ and $ Y=\cos(t_2)$ are in $[-1,1]$. 
	
	The function $f(X,Y)$ is nonnegative and is zero on the boundary of the square $[-1, 1]^2$. We find the maximal value of this function on the open square $(-1, 1)^2$ as follows. 
	
	We have
	\[
	\begin{cases} 
		\frac{\partial f}{\partial X} = 0  \\ 
		\frac{\partial f}{\partial Y}= 0
	\end{cases}
	\Longleftrightarrow 
	\begin{cases} 
		-3 x^3 - xy^2 + 4Ax^2y - (A^2 - 3)x - 2A y =0 \\ 
		-3 y^3 - x^2y + 4Ax y^2 - (A^2 - 3)y - 2A x =0.
	\end{cases}
	\]
	
	Now multiply the first by $Y$ and the second by $X$ and subtract, we get 
	$$(X^2 - Y^2)(-2 X Y + 2A) = 0.$$ 
	Since for every $a$ we have $A \ge 1$, it cannot happen that $X Y = A$.
	So $X=Y$  or $X=-Y$. We can easily show that $f(X,X)$ and $f(X,-X)$ are bounded by $\max\{4(A+1)^6, 16^2(A^2-1)^2\}$. Since $A$ varies from 1 to 3, these values are bounded by 16384. Thus, we have $\Delta_F \leq 16384$.
	
\end{proof}

\begin{lem}\label{19fields}
	There are 19 quadratic extensions $F$ (up to isomorphic)of complex quadratic fields of which the fundamental unit $\varepsilon$ generates $F$ and $|\varepsilon|< 1+ \sqrt{2}$.
\end{lem}
\begin{proof}
	Let $K$ be a complex quadratic subfield of $F$ with the discriminant $\Delta_K$. By Lemma \ref{bbddisc1}, we obtain that   $\Delta_F\leq 16384$. So,  we have  $|\Delta_K|  \leq 21$ (see Section 2 in \cite{ref:28} for more details). Using this and Ford's method in Section 5 and 6 in \cite{ref:28}, we can find all quadratic extensions of complex quadratic fields which have the discriminant at most 16384. Then by eliminating the case in which $|\varepsilon| \geq  1+ \sqrt{2}$ or $\varepsilon$ does not generate $F$ (see Lemma \ref{case3a}), we obtain 19 quartic fields listed in Table \ref{table1} below.
\end{proof}

In Table \ref{table1}, the second column contains the polynomials $P$ defining the quartic fields $F$ and the third column contains their regulators $R_F$. The fourth column shows the discriminant of some complex quadratic subfield $K $ of $F$. The seventh  column contains  upper bounds for $ \mathfrak{g}(s, \varepsilon)$ (see Lemma \ref{Rsmall}) when $s$ varies in the interval $[0.98,1/0.98]$. Note that computing an upper bound for $ \mathfrak{g}(s, \varepsilon)$ in Table \ref{table1}  is easy since it depends only on $s$ when $|\varepsilon|=e^{R_F/2}$ is given. The fifth and sixth columns are the cardinalities of the set $\mathfrak{B}_2(s)$ and $\mathfrak{B}_3(s)$ (that can be computed by using Lemma \ref{discK} and Remark \ref{reBj}).

\begin{center}
	\captionof{table}{}
	\label{table1}
	\resizebox{\textwidth}{!}{
		\begin{tabular}{|c|c | c |c|c |c|c|}
			\hline
			$$ &$ P $ & $R_F$ & $\Delta_K$ & $\#\mathfrak{B}_2(s)$ &  $\#\mathfrak{B}_3(s)$ & $\mathfrak{g}(s, \varepsilon)\leq$ \\ [0.5ex] 
			
			\hline 
			1& $x^4 - 3x^3 + 9$ & $0.5435$ & $-3$ & 0 & $0^*$ & $ -2.7 \cdot 10^{-6}$ \\      
			
			\hline 
			2 & $x^4 - x^3 + x + 1$ & $ 0.6330$ & $-7$ &  $6^*$ & 0 & $  -8.2\cdot 10^{-6}$ \\
			
			\hline 
			3 & $x^4 + 16x + 20$ & $0.7328$ & $-4$ & $0^*$ & 0 & $  -1.5\cdot 10^{-5}$ \\                                   
			
			\hline 
			4 &  $x^4 - x^3 + x^2 + x + 1$ & $0.7672$ & $-11$ &  0 & $0^*$ & $ -1.7 \cdot 10^{-5}$ \\
			
			\hline 
			5 & $	x^4 - x^3 + 2x + 1 $ & $0.8626$ & $-3$ & 0 & $0^*$ & $ -2.2 \cdot 10^{-5}$\\

			\hline 
			6 & $x^4+ 8x + 8   $ & $1.0613$ & $-4$ & $8^*$ & 0 & $ -2.6 \cdot 10^{-5}$ \\
			
			\hline 
			7 &  $x^4 - x^3 + 3x^2 + x + 1 $ & $ 1.1989$ & $-19$ &  0 & 0 & $  -2.6\cdot 10^{-5}$ \\
			
			\hline 
			8 & $x^4+36 $ & $ 1.3170$ & $-3; -4$ &  0 & 0 & $ -2.6 \cdot 10^{-5}$ \\
			\hspace*{0.2cm} & \hspace*{0.2cm} &  $\Delta_F = 144$ & \hspace*{0.2cm}& \hspace*{0.2cm}& \hspace*{0.2cm} & \hspace*{0.2cm} \\
			
			\hline 
			9 &  $x^4  + 4x^2 + 1 $ & $1.3170$  & $-24$ &  0 & 0 & $  -2.6 \cdot 10^{-5}$ \\
			\hspace*{0.2cm} & \hspace*{0.2cm} & $\Delta_F = 2304$ & \hspace*{0.2cm}& \hspace*{0.2cm}& \hspace*{0.2cm} & \hspace*{0.2cm} \\
			\hline 
			10 & $x^4 - x^3 + 4x^2 + x + 1 $ &  	$1.4290$   & $-23$ &  0 & 0 & $ -2.6 \cdot 10^{-5}$\\    
			
			\hline	
			11 & $x^4 - 3x^3 + 4x^2 + 1$  & 	$1.4608$   &  $-3$ & 0 & $0^*$ & $ -2.6 \cdot 10^{-5}$\\
			
			\hline
			12 & $x^4 +7 $ 		&	   	$1.4860$   & $-7$ &  $4^*$ & 0 & $ -2.6 \cdot 10^{-5}$\\
			
			\hline
			13 & $x^4 + 4x + 5 $  	&	   	$1.5286$   & $-4$ & $4^*$ & 0 & $ -2.6 \cdot 10^{-5}$\\
			
			\hline
			14 & $x^4 - x^3 - x^2 - 2x + 4$ & $1.5668$   & $-3; -7$ &  0 & 0 & $ -2.6 \cdot 10^{-5}$\\
			
			\hline
			15 & $x^4 +20$  		 	&	$1.6169$   &  $-20$ &  0 & 0 & $ -2.6 \cdot 10^{-5}$\\
			
			\hline
			16 &  $x^4 +3$ 			 & $1.6629$   & $-3$ &  0 & $6^*$ & $ -2.6 \cdot 10^{-5}$ \\
			
			\hline
			17 &  $x^4 - x^3 - 4x + 5$ & 		$1.6780$   & $-11$ &  0 & $0^*$ & $ -2.6 \cdot 10^{-5}$ \\
			
			\hline
			18 &  $x^4 - x^3 + 4x^2 - 6x + 3$ &   	$1.7366$   & $-3$ &  0 & $0^*$ & $ -2.6 \cdot 10^{-5}$ \\
			
			\hline 
			19 & $x^4 +135			$ & $	1.7400$ & $-15$ &  0 & 0 & $  -2.6 \cdot 10^{-5}$\\      
			\hline
		\end{tabular}
	}
\end{center}

\begin{lem}\label{discK} 
	Let $F$ be a quadratic extension of a complex quadratic subfield $K$ and let $\Delta_K$ be the discriminant of $K$. Assume that $\mathfrak{B}_2(s)$ or $\mathfrak{B}_3(s)$ is nonempty. 
	Then $\Delta_K \in \{-3, -4,-7,-11 \}$. Moreover, if $\Delta_K \in \{-3,-11 \}$ then $\mathfrak{B}_2(s) = \emptyset$ and if $\Delta_K \in \{-4,-7 \}$ then $\mathfrak{B}_3(s) = \emptyset$.	
\end{lem}

\begin{proof}
	Assume that $\mathfrak{B}_2(s)$ or $\mathfrak{B}_3(s)$ is nonempty. Then there is an element $x$ of $O_F$ of norm $N_{F/\mathbb{Q}}(x) \in \{2,3\}$. So the element $y= N_{F/K}(x) \in O_K$ also has norm $2$ or $3$. 
	This means that there are some $a,b \in \mathbb{Z}$ such that $a^2 + |\Delta_K|b^2 \in \{8, 12\}$. It follows that $|\Delta_K|$ is at most $12$. 
	So the possible values of $\Delta_K $ are $-3,-4,-7,-8$ and $-11$. 
	For $\Delta_K \in \{-3, -11\}$, the prime $2$ is inert, so there are no elements of norm $2$. In other words, we get $\mathfrak{B}_2(s) = \emptyset$. 
	For $\Delta_K \in \{-4, -7\}$,  the prime $3$ is inert, so $\mathfrak{B}_3(s) = \emptyset$. 
\end{proof}


\begin{rem}\label{reBj}
	Let $F$ be a quadratic extension of a complex quadratic subfield $K$ and let $\Delta_K$ be the discriminant of $K$. 
	By this lemma, we can check whether $F$ has $\mathfrak{B}_2(s) = \emptyset$ or $\mathfrak{B}_3(s) = \emptyset$ by checking if the value of  $\Delta_K$ is in the set $\{-3, -4,-7,-11 \}$ (and this can be easily tested by using \texttt{sage}). For example, the first quartic field in Table \ref{table1} contains a complex quadratic subfield $K$ with $\Delta_K =-3$, so we have $\mathfrak{B}_2(s) = \emptyset$ and since the seventh quartic field in Table \ref{table1} contains a complex quadratic field $K$ with $\Delta_K =-19$, so we have $\mathfrak{B}_2(s) = \mathfrak{B}_3(s)=\emptyset$. 
	
	However, in some cases, the discriminant $\Delta_K$ does not show whether $\mathfrak{B}_2(s)$ or $\mathfrak{B}_3(s)$ is empty. For instance, for the first number field in Table \ref{table1}, we do not know how many elements $\mathfrak{B}_3(s)$ has.
	There are 12 such cases (marked with $*$ in Table \ref{table1}). So, we have to compute $\#\mathfrak{B}_2(s)$ for the quartic fields 2, 3, 6, 12 and 13 and compute $\#\mathfrak{B}_3(s)$ for the quartic fields 1, 4, 5, 11, 16, 17 and 18 in Table \ref{table1}. 
	
	For these quartic fields, to count the number of elements of $\mathfrak{B}_2(s)$ and $\mathfrak{B}_3(s)$, we first find an LLL-reduced basis $\{ b_1, b_2, b_3, b_4\}$  of the lattice $O_F$. Let $x \in \mathfrak{B}_j(s)$ with $j =2, 3$. Then $x= s_1 b_1 + s_2 b_2 + s_3 b_3 + s_4 b_4 $ for some integers $s_1, s_2, s_3, s_4$. By Lemma \ref{lengthB}, we have  $\|x\|^2 < 11$. Since $\|b_1\| \geq \|1\| =2$, we have
	$$|s_i| \leq 2^{3/2} (3/2)^{4-i} \frac{\|x\|}{\|b_1\|} \leq 2^{3/2} (3/2)^{4-i} \frac{\sqrt{9.2}}{2} \text{ for all } i =1, 2, 3, 4.$$
	See Section 12 in \cite{ref:1}.
	So $$|s_1| \leq 15, \hspace*{1cm}|s_2| \leq 10,\hspace*{1cm} |s_3| \leq 7 \hspace*{0.5cm}\text { and }\hspace*{0.5cm} |s_4| \leq 4.$$
	By computing $16 \cdot 21 \cdot 15 \cdot 9 = 45360$ possibilities of $x$ (up to sign) obtained from these values of $s_1, s_2, s_3, s_4$, then checking their norms,  we can easily obtain the cardinality of $\mathfrak{B}_j(s)$.\\
	Another method to compute the cardinalities of $\mathfrak{B}_2(s)$ and $\mathfrak{B}_3(s)$ is the Fincke--Pohst 
	algorithm {\cite[Algorithm 2.12]{ref:40}} that is implemented in \texttt{pari-gp} by the function \texttt{qfminim}.  
\end{rem}

Denote by  
$$\mathfrak{g}(s, \varepsilon)= \omega \cdot(G(s,1) + G(s, \varepsilon) + G(s, \varepsilon^{-1})+ G(s, \varepsilon^2) + G(s, \varepsilon^{-2})).$$ 
\begin{lem}\label{Rsmall}
	If $F$ satisfies the following conditions.
	\begin{itemize}
		\item[i)] $R_F > 0.54$, 
		\item[ii)] For each $s \in [0.98, 1/0.98]$, we have $\#\mathfrak{B}_j(s) \leq 30$ for $j=2,3$ and
		\item[iii)] For all $s \in [0.98, 1/0.98]$, we have 
		$\mathfrak{g}(s, \varepsilon) \leq -2.6 \cdot 10^{-6}, $ 
	\end{itemize} 
	then  $g''(s) <0$ for all   $s \in [0.98, 1/0.98]$.
\end{lem}

\begin{proof}
	Since $R_F > 0.54$, we have $\|\varepsilon^m u\|^2 \geq 4 \sqrt{4}$ for all integers $|m| \geq 3$ and $s \in [0.98, 1/0.98]$. Thus, if $x \in B_i$ for $i =1, 2, 3$ and $N(x) =1$ then $x \in \{ \zeta, \zeta \cdot \varepsilon, \zeta \cdot \varepsilon^{-1} , \zeta \cdot \varepsilon^2, \zeta \cdot \varepsilon^{-2} \}$ where $\zeta$ runs over the roots of unity of $F$. This and the fact that $ 1 \in B_1$ together with  Lemma \ref{pos} imply that
	
	$$T_1 +T_2+T_3 \leq \mathfrak{g}(s, \varepsilon) + \sum_{x \in B_2, N(x) \neq 1} G(s,x) + \sum_{x \in B_3, N(x) \neq 1} G(s,x).$$
	
	Using condition ii) and an argument similar to the proof of Lemma \ref{lemT2}, \ref{lemT3}, we obtain that $\sum_{x \in B_2, N(x) \neq 1} G(s,x) \leq 1.65 \cdot 10^{-6}$ and\\
	$\sum_{x \in B_3, N(x) \neq 1} G(s,x) \leq 5.2 \cdot 10^{-7}$. 
	
	By assumption iii), we get
	$\mathfrak{g}(s, \varepsilon)\leq -2.6 \cdot 10^{-6}$. Moreover, Lemma \ref{lemT4} says that $T_4 \leq 3.9 \cdot 10^{-7}$. Since $g''(s) = T_1 +T_2 + T_3 + T_4 $, the result follows. 
	
\end{proof}

Now we prove Proposition \ref{pro3b}.
\begin{proof}
	Lemma \ref{19fields} says that there are only 19 quartic fields satisfying the conditions of Proposition \ref{pro3b}. They are given in Table \ref{table1}. 
	
	We can prove that in this case, $h^0$ has its unique global maximum at $D_0$ in 3 steps (see the proof of Proposition \ref{pro3a}). The readers can easily check Step 1 and Step 2 and see the maximum of $h^0$ in Figure \ref{Rverysmall}, \ref{68111619}, \ref{9101518} and \ref{712171314}. In these figures, $h^0$ is periodic and the period is the regulator of the number field. 
	
	Here we only prove Step 3. In other words, we prove that $h^0$ has its local maximum at $D_0$ on $\To$. 
	
	We have known that $g'(1) =0$ (see the proof of Proposition \ref{pro2c}). So it is sufficient to prove that $g''(s)<0$ on the interval $[0.98, 1/0.98]$.  By Lemma \ref{Rsmall}, this can be done by checking 3 conditions i), ii) and iii). Table \ref{table1} shows that all 19 number fields satisfy these conditions. Therefore, the result follows.
\end{proof}

\begin{figure}[h]
	\centering
	\includegraphics[width=0.6\linewidth]{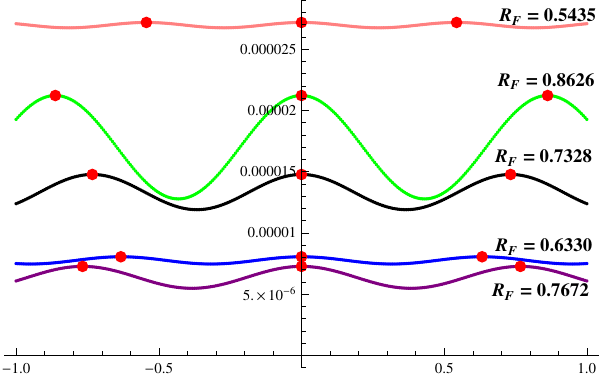}
	\caption{{\hspace*{0.5cm}} 	\label{Rverysmall}}
\end{figure}   

\begin{figure}[h]
	\centering
	\includegraphics[width=0.6\linewidth]{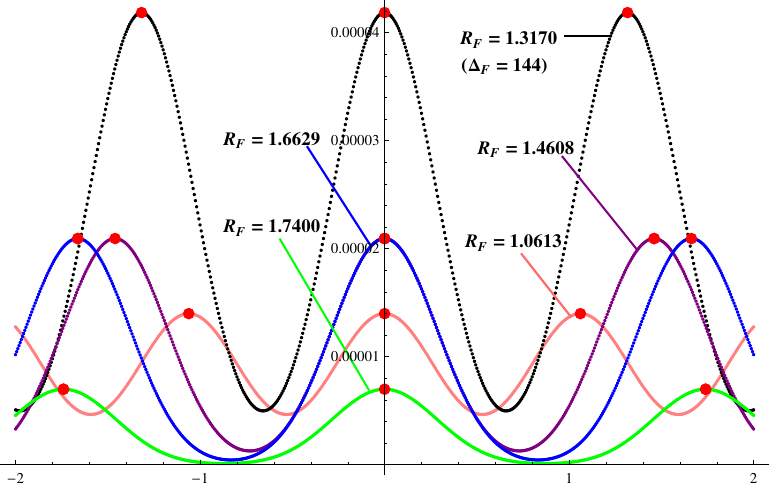}
	\caption{{\hspace*{0.5cm}}	\label{68111619}}
\end{figure}

\begin{figure}[h]
	\centering
	\includegraphics[width=0.6\linewidth]{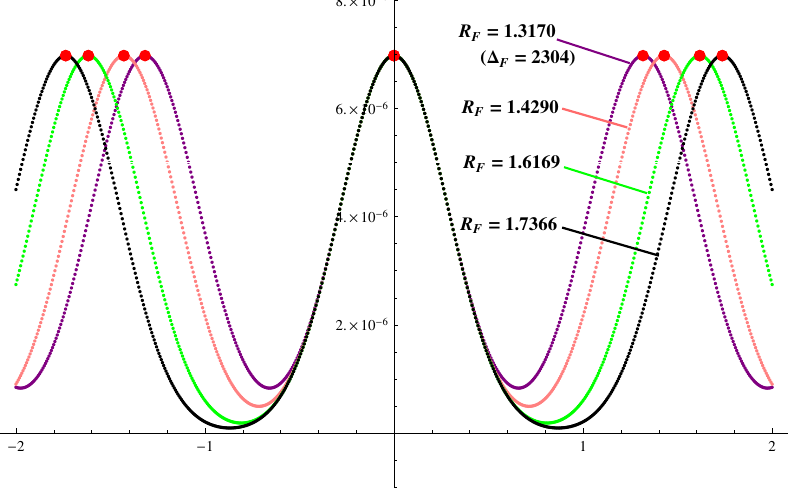}
	\caption{{\hspace*{0.5cm}}	\label{9101518}}
\end{figure}

\begin{figure}[h]
	\centering
	\includegraphics[width=0.6\linewidth]{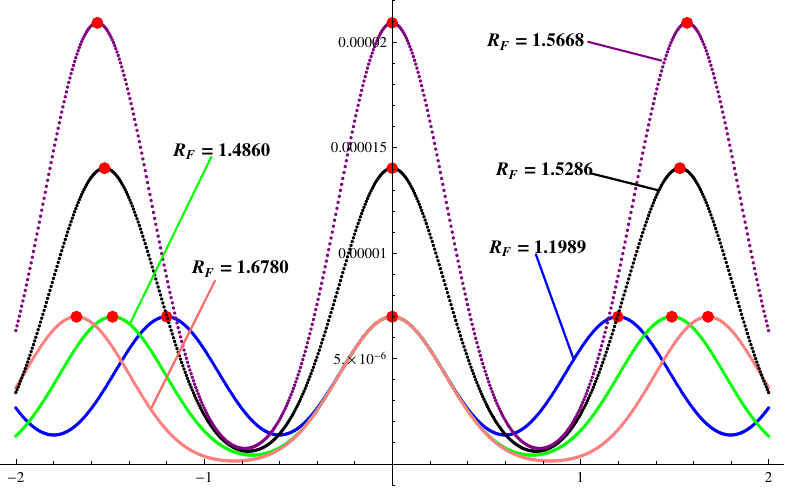}
	\caption{{\hspace*{0.5cm}}	\label{712171314}}
\end{figure}

\newpage
\section*{Acknowledgement}
I would like to thank  Ren\'{e} Schoof for discussion and very valuable comments. I also would like to thank Wen-Ching Li and the National Center for Theoretical Sciences (NCTS) for supporting and hospitality during the fall 2014.\\
This research was partially supported by the Academy of Finland (grants  $\#$276031, $\#$282938, and $\#$283262). The support from the European Science Foundation under the COST Action IC1104 is also gratefully acknowledged. 



\begin{thebibliography}{10}
	
	\bibitem{ref:0}
	{\sc E.~Bayer-Fluckiger}, {\em Lattices and number fields}, in Algebraic
	geometry: {H}irzebruch 70 ({W}arsaw, 1998), vol.~241 of Contemp. Math., Amer.
	Math. Soc., Providence, RI, 1999, pp.~69--84.
	
	\bibitem{ref:11}
	{\sc H.~Cohen}, {\em A course in computational algebraic number theory},
	vol.~138 of Graduate Texts in Mathematics, Springer-Verlag, Berlin, 1993.
	
	\bibitem{ref:40}
	{\sc U.~Fincke and M.~Pohst}, {\em Improved methods for calculating vectors of
		short length in a lattice, including a complexity analysis}, Math. Comp., 44
	(1985), pp.~463--471.
	
	\bibitem{ref:28}
	{\sc D.~Ford}, {\em Enumeration of totally complex quartic fields of small
		discriminant}, in Computational number theory ({D}ebrecen, 1989), de Gruyter,
	Berlin, 1991, pp.~129--138.
	
	\bibitem{ref:14}
	{\sc P.~Francini}, {\em The size function {$h^0$} for quadratic number fields},
	J. Th\'eor. Nombres Bordeaux, 13 (2001), pp.~125--135.
	\newblock 21st Journ{\'e}es Arithm{\'e}tiques (Rome, 2001).
	
	\bibitem{ref:15}
	\leavevmode\vrule height 2pt depth -1.6pt width 23pt, {\em The size function
		{$h^\circ$} for a pure cubic field}, Acta Arith., 111 (2004), pp.~225--237.
	
	\bibitem{ref:27}
	{\sc R.~P. Groenewegen}, {\em The size function for number fields}.
	\newblock Doctoraalscriptie, Universiteit van Amsterdam, 1999.
	
	\bibitem{ref:21}
	\leavevmode\vrule height 2pt depth -1.6pt width 23pt, {\em An arithmetic
		analogue of {C}lifford's theorem}, J. Th\'eor. Nombres Bordeaux, 13 (2001),
	pp.~143--156.
	\newblock 21st Journ{\'e}es Arithm{\'e}tiques (Rome, 2001).
	
	\bibitem{ref:1}
	{\sc H.~W. Lenstra, Jr.}, {\em Lattices}, in Algorithmic number theory:
	lattices, number fields, curves and cryptography, vol.~44 of Math. Sci. Res.
	Inst. Publ., Cambridge Univ. Press, Cambridge, 2008, pp.~127--181.
	
	\bibitem{ref:29}
	{\sc M.~Pohst}, {\em Regulatorabsch\"atzungen f\"ur total reelle algebraische
		{Z}ahlk\"orper}, J. Number Theory, 9 (1977), pp.~459--492.
	
	\bibitem{ref:4}
	{\sc R.~Schoof}, {\em Computing {A}rakelov class groups}, in Algorithmic number
	theory: lattices, number fields, curves and cryptography, vol.~44 of Math.
	Sci. Res. Inst. Publ., Cambridge Univ. Press, Cambridge, 2008, pp.~447--495.
	
	\bibitem{ref:3}
	{\sc G.~van~der Geer and R.~Schoof}, {\em Effectivity of {A}rakelov divisors
		and the theta divisor of a number field}, Selecta Math. (N.S.), 6 (2000),
	pp.~377--398.
	
\end{thebibliography}
\end{document}